%% file: ArticleArXiv.tex
\documentclass[12pt]{amsart}

\usepackage{pifont,geometry}
\usepackage{amsmath,amssymb,amsthm}
\usepackage[latin1]{inputenc}
\usepackage[english]{babel}
\usepackage[pdftex,bookmarks,colorlinks]{hyperref}

\usepackage{graphicx}
\usepackage{enumerate}

\usepackage[x11names,rgb,table]{xcolor}

\newtheorem{theorem}{Theorem}[section]

\newtheorem{proposition}[theorem]{Proposition}
\newtheorem{lemma}[theorem]{Lemma}

\theoremstyle{definition}
\newtheorem{definition}[theorem]{Definition}

\newtheorem{remark}[theorem]{Remark}

\numberwithin{equation}{section}

\theoremstyle{theorem}

\newcommand{\K}{K}
\newcommand{\PP}{\mathbb{P}}

\newcommand{\FQ}{\mathbb{F}_{q^2}}
\newcommand{\Fq}{\mathbb{F}_q}
\newcommand{\Fqs}{\mathbb{F}_{q^u}}

\newcommand{\Tr}[3]{\mathrm{Tr}^{#1}_{#2}(#3)}

\newcommand{\N}[3]{\mathrm{N}^{#1}_{#2}(#3)}

\newcommand{\He}{\mathcal{H}}
\newcommand{\Bmq}{\mathcal{B}_{m,q}}
\newcommand{\A}{\mathcal{A}}
\newcommand{\X}{\mathcal{X}}
\newcommand{\Y}{\mathcal{Y}}
\newcommand{\LL}{\mathcal{L}}

\begin{document}

\title{\normalfont{Higher Hamming weights for locally recoverable codes on algebraic curves}}

\author{Edoardo Ballico}
\address{\textnormal{Edoardo Ballico}. Dipartimento di Matematica dell'Universit\`{a} di Trento\\ 
         Via Sommarive 14, 
         38123 Povo (TN), Italy}
\email{{edoardo.ballico@unitn.it}}

\author{Chiara Marcolla}
\address{\textnormal{Chiara Marcolla}. Dipartimento di Matematica dell'Universit\`{a} di Torino\\ 
         Via Carlo Alberto 10, 
         10123 Torino, Italy}
\email{{chiara.marcolla@unito.it}}



\maketitle 

\begin{abstract} 

We study the locally recoverable codes on algebraic curves. In the first part of this article, we provide a bound of generalized Hamming weight of these codes. Whereas in the second part, we propose a new family of algebraic geometric LRC codes, that are LRC codes from Norm-Trace curve. Finally, using some properties of Hermitian codes, we improve the bounds of distance proposed in \cite{btv} for some Hermitian LRC codes.

\end{abstract}


\section{Introduction}
\label{intro}

\input{intro}


\section{Preliminary notions}\label{pre}
\subsection{Algebraic geometric codes}
\label{AG}
\input{AG}


\subsection{Algebraic geometric locally recoverable codes}
\label{AG_LRC}
\input{AG_LRC}


\section{Generalized Hamming weights of AG LRC codes}
\label{weight_AG_LRC}
\input{weight_AG_LRC}


\section{LRC codes from Norm--Trace Curve}
\label{NT_LRC_code}
\input{NT_LRC_code}


\section{Hermitian LRC codes}
\label{Her_Dist}
\input{Her_Dist}

\section*{Acknowledgements}
The first author is partially supported by MIUR and GNSAGA
of INDAM (Italy).\\

The authors would like to thank the anonymous referees for their  comments.

\bibliographystyle{abbrv}

\bibliography{RefsBallico}
\end{document}

%% file: intro.tex
The $v$-th generalized Hamming weight $d_v(C)$ of a linear code $C$ is the minimum support size of $v$-dimensional
subcodes of $C$. The sequence $d_1(C),\ldots,d_k(C)$ of generalized Hamming weights 
was introduced by Wei \cite{wei1991generalized} to characterize the performance of a linear code on the wire-tap channel of type II. 
Later, the GHWs of linear codes have been used in many other applications regarding the communications, as for bounding the covering radius of linear codes \cite{janwa2007generalized}, 
in network coding \cite{ngai2011network},  in the context of list decoding \cite{gopalan2011list,guruswami2003list},  and finally for secure secret sharing \cite{kurihara2011strongly}.
Moreover, in \cite{chen2007secure} the authors show in which way an arbitrary linear code gives rise to a secret sharing scheme, in \cite{kasami1993complexity, kasami1993optimum} 
the connection between the trellis or state complexity of a code and its GHWs is found and in \cite{forney1994dimension} the author proves the equivalence to the dimension/length profile of a code  and its generalized Hamming weight. 
For these reasons, the GHWs (and their \textit{extended} version, the \textit{relative} generalized Hamming weights \cite{luo2005some, kurihara2012secret}) play a central role in coding theory.
In particular, generalized and relative generalized Hamming weights are studied for Reed-Muller codes \cite{heijnen1998generalized,  martin2014relative} and for codes constructed by using an algebraic curve \cite{geil2014relative} as Goppa codes \cite{munuera1994generalized, yang1994weight}, Hermitian codes  \cite{homma2009second, munuera1999second} and Castle codes \cite{olaya2015second}.\\



In this paper, we provide a bound on the generalized Hamming weight of locally recoverable codes on the algebraic curves proposed in \cite{btv}. Moreover, we introduce a new family of algebraic geometric LRC codes and improve the bounds on the distance for some Hermitian LRC codes.\\

Locally recoverable codes were introduced in \cite{9gopalan2012locality} and they have been significantly studied because of their
applications in distributed and cloud storage systems \cite{7forbes2014locality, 11huang2013pyramid, 7Dsilberstein2013optimal,9Dtamo2014family, 11Dtamo2013optimal}.
We recall that a code $C\in (\Fq)^n$
has locality $r$ if every symbol of a codeword $c$ can be recovered from a subset of $r$ other symbols of $c$.

\noindent In other words, we consider a finite field $\K=\Fq$, where $q$ is a power of a prime, and an $[n,k]$ code  $C$ over the field $\K$, where $k = log _q(|C|)$. For each $i\in \{1,\ldots,n\}$ and each $a\in \K$ set $C(i,a)= \{c\in C \mid c_i=a\}$.
For each $I\subseteq \{1,\ldots,n\}$ and each $S\subseteq C$ let $S_I$ be the restriction of $S$ to the coordinates in $I$.
\begin{definition}
Let $C$  be an $[n,k]$ code over the field $\K$, where $k = log _q(|C|)$. Then $C$ is said to have {\textbf{all-symbol locality r}} if for each $a\in \Fq$ and each $i\in \{1,\ldots,n\}$ there is $I_i\subset\{1,\ldots,n\}\setminus \{i\}$ with $|I_i|\leq r$,
such that for $C_{I_i}(i,a)\cap C_{I_i}(i,a') =\emptyset$ for all $a\ne a'$. We use the notation $(n,k,r)$ to refer to the parameters of this code.
\end{definition}
Note that if we receive a codeword $c$ correct except for an erasure at $i$, we can recover the codeword by looking at its coordinates in $I_i$. For this reason, $I_i$ is called a \textit{recovering set} for the
symbol $c_i$.\\

Let $C$ be an $(n,k,r)$ code, then the distance of this code has to verify the bound proved in  \cite{3papailiopoulos2014locally,9gopalan2012locality} that is $ d\leq n-k- \lceil k/r\rceil +2$.
The codes  that achieve this bound with equality are called \textit{optimal} LRC codes \cite{7Dsilberstein2013optimal, 9Dtamo2014family, 11Dtamo2013optimal}.  Note that when $r = k$, we obtain the Singleton bound, therefore optimal LRC codes with $r=k$ are MDS codes.

 %

%
\paragraph*{Layout of the paper} This  paper is divided as follows. In Section~\ref{pre} we recall the notions of algebraic geometric codes and the definition of algebraic geometric locally recoverable codes introduced in \cite{btv}.
In Section~\ref{weight_AG_LRC} we provide a bound on the generalized Hamming weights of the latter codes. In Section~\ref{NT_LRC_code} we propose a new family of algebraic geometric LRC codes, which are LRC codes from the Norm--Trace curve. Finally, in Section \ref{Her_Dist} we improve the bounds on the distance proposed in \cite{btv} for some Hermitian LRC codes, using some properties of the Hermitian codes.

%% file: AG.tex
Let $\K=\Fq$ be a finite field, where $q$ is a power of a prime.
Let $\X$ be a smooth projective absolutely irreducible nonsingular  curve over $\K$.
We denote by $\K(\X)$ the rational functions field on $\X$.
%
Let $D$ be a divisor on the curve $\X$. We recall that the \textit{Riemann-Roch space} associated to $D$ is a vector space
$\mathcal{L}(D)$ over $\K$ defined as
$$\mathcal{L}(D) = \{f \in \K(\X) \mid (f) + D \geq 0\} \cup \{0\}.$$
where we denote by $(f)$ the divisor of $f$.\\

Assume that $P_1,\ldots,P_n$ are rational points on $\mathcal{X}$ and $D$ is a divisor such that $D=P_1+\ldots+P_n.$
Let $G$ be some other divisor such that $supp(D)\cap supp(G)=\emptyset$.
Then we can define the algebraic geometric code as follows:
\begin{definition}
The \textbf{algebraic geometric code} (or AG code) $C(D,G)$ associated
with the divisors $D$ and $G$ is defined as
$$C(D,G) = \{(f(P_1), \ldots, f(P_n)) \mid f\in \mathcal{L}(G)\} \subset \K^n.$$
The dual  $C^\perp(D,G)$ of $C(D,G)$ is an algebraic geometric code.
\end{definition}
\noindent In other words an algebraic geometric code is the image of the evaluation map $\mathrm{Im}(ev_{D})=C(D,G)$, where the \textit{evaluation map}  $ev_D : \mathcal{L}(G)\rightarrow \K^n$ is given by
$$ev_D(f) = (f(P_1),\ldots, f(P_n))\in \K^n.$$
Note that if  $D=P_1+\ldots+P_n$ and we denote by $\mathcal P=\{P_1,\ldots,P_n\}$ we can also indicate $ev_D$ as $ev_\mathcal{P}$.

%% file: AG_LRC.tex
In this section we consider the construction of algebraic geometric locally recoverable codes of \cite{btv}. \\

%
Let $\X$ and $\Y$  be smooth projective absolutely irreducible curves over $\K$.
Let $g : \X \rightarrow \Y$ be a rational separable
map of curves of degree $r + 1$. Since $g$ is separable, then there exists a function  $x\in\K(\X)$ such that $\K(\X)=\K(\Y)(x)$ and that $x$ satisfies the equation $x^{r+1} + b_rx^r +\ldots+ b_0 = 0$, where $b_i\in \K(\Y)$. The function $x$ can be considered as a map $x : \X \rightarrow \PP_\K$. Let $h=\deg(x)$ be the degree of $x$.\\
We consider a subset $S=\{P_1,\ldots,P_s\}\subset \Y(\K)$ of $\Fq$-rational points of $\Y$, a divisor $Q_\infty$ such that $supp(Q_\infty)\cap supp(S)=\emptyset$ and a positive divisor $D=t Q_\infty$. We denote by 
$$
\A=g^{-1}(S)=\{P_{ij}, \mbox{ where } i=0,\ldots,r, \, j=1,\ldots,s\}\subset\X(\K),
$$ 
where $g(P_{ij})=P_i$ for all $i,j$ and assume that $b_i$ are functions in $\LL(n_iQ_\infty)$ for some natural numbers $n_i$ with $i=1,\ldots,r$.\\
Let $\{f_1,\ldots,f_m\}$ be a basis of the Riemann-Roch space $\LL(D)$. By the Riemann-Roch Theorem we have that $m\geq \deg(D)+1-g_\Y$, where $g_\Y$ is the genus of $\Y$.\\

From now on, we assume that $m=\deg(D)+1-g_\Y$, where $\deg(D)=t\ell$, and we consider the $\K$-subspace $V$ of $\K(\X)$ of dimension $rm$ generated by
$$\mathcal B=\{f_jx^i, \, i=0,\ldots,r-1,\, j=1,\ldots, m\}.$$
We consider the evaluation map $ev_{\mathcal{A}}:V\rightarrow \K^{(r+1)s}$.
Then we have the following theorem.
\begin{theorem}\label{teo.3.1}
The linear space  $C(D,g)=\mathrm{Span}_{\K^{(r+1)s}}\langle ev_{\A}(\mathcal B)\rangle$ is an $(n,k,r)$ algebraic geometric LRC code with parameters
$$
\begin{array}{ll}
   n&=(r+1)s\\
k&=rm\geq r(t\ell+1-g_\Y) \\
d&\geq n-t\ell(r+1)-(r-1)h.
\end{array}
$$ 
\begin{proof} 
See Theorem 3.1 of \cite{btv}.
\end{proof}
\end{theorem}
\noindent The AG LRC codes have an additional property.
They are LRC codes $(n,k,r)$ with $(r+1)\,|\, n$ and $r\,|\,k$. The set $\{1,\ldots,n\}$ can be divided into $n/ (r+1)$ disjoint subsets $U_j$ for $1\le j \le s$ with the same cardinality $r+1$.
For each $i$ the set $I_i\subseteq \{1,\ldots,n\}\setminus\{i\}$ is the complement of $i$ in the element of the partition $U_j$ containing $j$, i.e.
for all $i, j\in \{1,\ldots,n\}$ either $I_i =I_j$ or $I_i\cap I_j=\emptyset$.\\ 
Moreover, they have also the following nice property. Fix $w\in (\K)^n$ and denote by
$w_{U_j}=\{w_{\iota}$, for any $\iota \in U_j\}$. Suppose we receive all the symbols in $U_j$. There is a simple linear parity test on the $r+1$ symbols of $U_j$ such that if this parity check fails
we know that at least one of the symbols in $U_j$ is wrong. If we are guaranteed (or we assume) that at most one of the symbols in $U_j$
is wrong and the parity check is OK, then all the symbols in $U_j$ are correct.
Moreover we can recover an erased symbol $w_\iota$, with $\iota\in U_j$ using a polynomial interpolation through the points of the recovering set $w_{U_j}$.

%% file: weight_AG_LRC.tex
Let $\K$ be a field and let $\X$ be a smooth and geometrically connected curve of genus $g\ge 2$ defined over the field $K$.  We also assume $\X(\K)\ne \emptyset$. 
We recall the following definitions:
\begin{definition}[\cite{p}, \cite{p2}]
The $\K$-\textbf{gonality} $\gamma_\K(\X)$ of $\X$ over a field $\K$ is the smallest possible degree of a dominant rational map 
$\X \rightarrow \PP^1_\K$. For any field extension $L$ of $\K$, we define also the $L$\textbf{-gonality}  $\gamma_L(\X)$ of $\X$ as the gonality of the base extension $\X_L = \X \times_\K L$. It is an invariant of the function field $L(\X)$ of $\X_L$.
\end{definition}
Moreover, for each integer $i>0$, the \textit{$i$-th gonality} $\gamma _{i,L}(\X)$ of $\X$ 
is the minimal degree $z$ such that there
is $R\in \mbox{Pic}^z(\X)(L)$ with $h^0({R}) \ge i+1$. The sequence $\gamma _{i,\overline{\K}}(\X)$ is the usual \textit{gonality sequence} \cite{lm}.
Moreover, the integer $\gamma _{1,\K}(\X)=\gamma_\K(\X)$ is the  $\K$-gonality of $\X$. \\

Let $\K=\Fq$ a finite field with $q$ elements. 
Let $C\subset \K^n$ be a linear $[n,k]$ code over $\K$. 
We recall that the \textit{support} of $C$ is defined as follows
$$
supp(C)=\{i \mid c_i \ne 0 \mbox{ for some } c \in C\}.
$$
So $\sharp supp(C)$ is the number of nonzero columns in a generator matrix for $C$.
Moreover, for any $1\leq v\leq k$, the \textit{$v$-th generalized Hamming weight} of $C$  \cite[\S 7.10]{hp}, \cite[ \S 1.1]{tvn} is defined by 
$$
d_v(C)=\min\{\sharp supp(\mathcal D) \mid \mathcal D \mbox{ is a linear subcode of $C$ with } dim(\mathcal D)=v\}.
$$
In other words, for any integer $1\leq v\le k$, $d_v(C)$ is the $v$-th minimum support weights, i.e. the minimal integer $t$ such that
there are an $[n,v]$ subcode $\mathcal {D}$ of $C$ and a subset $S\subset \{1,\dots ,n\}$ such that $\sharp (S)=t$ and each codeword of $\mathcal {D}$ has zero coordinates outside $S$. The sequence $d_1(C),\ldots,d_k(C)$ of generalized Hamming weights (also called \textit{weight hierarchy} of $C$) is strictly increasing (see Theorem~7.10.1 of \cite{hp}).
Note that $d_1(C)$ is the minimum distance of the code $C$.\\

Let us consider $\X$ and $\Y$  smooth projective absolutely irreducible curves over $\K$ and let $g : \X \rightarrow \Y$ be a rational separable
map of curves of degree $r + 1$. Moreover we take $r,t,Q_\infty$, $f_1,\dots ,f_m$ and $\A= g^{-1}(S)$ defined as Section \ref{AG_LRC}. So we can construct  an $(n,k,r)$ algebraic geometric LRC code $C$ as in Theorem \ref{teo.3.1}. For this code we have the following:
\begin{theorem}\label{a1}
Let $C$ be an $(n,k,r)$ algebraic geometric LRC code as in Theorem~\ref{teo.3.1}.
For every integer $v \ge 2$ we have that
$$d_v(C) \ge n-t\ell (r+1) -(r-1)h +\gamma _{v-1,K}(\X).$$
\end{theorem}

\begin{proof}
 Take a $v$-dimensional linear subspace $\mathcal {D}$ of $C$ and call
$$E\subseteq \{P_{ij} \mid i=0,\dots r,\, j=1,\dots ,s\},$$ 
the set of common zeros of all elements of $\mathcal {D}$. 
Since $n-d_v(C)=\sharp (E)$, we have to prove that $t\ell (r+1)+(r-1)h -\sharp (E)\ge \gamma _{v-1,K}(X)$. Fix $u\in \mathcal {D}\setminus \{0\}$ and let $F_u$ denote the zeros of $u$. Note that $F_u$ is contained
in the set $\{P_{ij} \mid i=0,\dots r, \,j=1,\dots ,s\}$ by the definition of the code $C$. We have
$F_u\supseteq E$. By the definition of the integers $t, \ell$ and $h:= \deg (x)$, we have $\sharp (F_u) \le t\ell (r+1)+(r-1)h$. 
The divisors $F_u-E$, $u\in \mathcal {D}\setminus \{0\}$ form a family of linearly equivalent non-negative divisors, each of them defined over $\K$. Since
$\dim (\mathcal {D}) = v$, the definition of $\gamma _{v-1,\overline{\K}}(\X)$ gives $\sharp (F_u)-\sharp (E) \ge \gamma _{v-1,\K}(\X)$. This inequality
for a single $u\in \mathcal {D}\setminus \{0\}$ proves the theorem.
\end{proof}

See Remark \ref{rem.NT} for an application of Theorem \ref{a1}.


%% file: NT_LRC_code.tex
In this section we propose a new family of Algebraic Geometric LRC codes, that is, a LRC codes from the Norm--Trace curve. Moreover, we compute the $\Fqs$-gonality of the Norm-Trace curve. \\

Let $\K=\Fqs$ be a finite field, where $q$ is a power of a prime.
We consider the \textit{norm} $\mathrm{N}^{\mathbb{F}_{q^u}}_{\Fq}$ and the \textit{trace}
$\mathrm{Tr}^{\mathbb{F}_{q^u}}_{\Fq}$, two functions from $\mathbb{F}_{q^u}$ to $\Fq$ defined as
$$\N{\mathbb{F}_{q^u}}{\Fq}{x}=x^{1+q+\dots+q^{u-1}} \mbox{ and }\,\,
\Tr{\mathbb{F}_{q^u}}{\Fq}{x}=x+x^{q}+\dots+x^{q^{u-1}}.$$
\noindent The \textit{Norm-Trace curve} $\chi$ is the curve
defined over $\K$ by the following affine equation 
$$\N{\mathbb{F}_{q^u}}{\Fq}{x}=\Tr{\mathbb{F}_{q^u}}{\Fq}{y},$$
that is,
\begin{equation}\label{eqNT}
x^{(q^u-1)/(q-1)}=y^{q^{u-1}}+y^{q^{u-2}}+\ldots+ y \mbox{ where } x,y\in\K
\end{equation}
The Norm-Trace curve $\chi$ has exactly $n=q^{2u-1}$ $\K$-rational affine points (see Appendix A of \cite{CGC-cd-art-Geil-1}), that we denote by $\mathcal{P}_\chi=\{P_1,\ldots,P_n\}$.  The genus of $\chi$ is $g =\frac{1}{2}(q^{u-1}-1)(\frac{q^u-1}{q-1}-1)$.
Note that if we consider $u=2$, we obtain the Hermitian curve. \\

Starting from the Norm--Trace curve, we have two different ways to construct Norm--Trace LRC codes.
\paragraph{\textbf{Projection on $\mathbf x$}}
We have to construct a $q^u$-ary $(n,k,r)$ LRC codes.
We consider the natural projection $g(x,y)=x$. Then the degree of $g$ is $q^{u-1}=r+1$ and the degree of $y$ is $h=1+q+\dots+q^{u-1}$.\\
To construct the codes we consider $S=\Fqs$  and $D=tQ_\infty$ for some $t\geq 1$. Then, using a construction of Theorem \ref{teo.3.1} we find the parameters for these  Norm--Trace LRC codes.
\begin{proposition}
A family of Norm--Trace LRC codes has the following parameters:
$$n=q^{2u-1}, \quad k=mr=(t+1)(q^{u-1}-1)$$ and 
$$d\geq n-tq^{u-1}-(q^{u-1}-1)(1+q+\dots+q^{u-1}).$$
\end{proposition}
\paragraph{\textbf{Projection on $\mathbf y$}}
We have to construct a $q^u$-ary $(n,k,r)$ LRC codes.
We consider the other natural projection $g'(x,y)=y$. Then $\deg(g')=1+q+\dots+q^{u-1}=r+1$.
In this case we take $S=\Fqs\backslash M$, where 
$$M=\{a\in\Fqs \mid a^{q^{u-1}}+a^{q^{u-2}}+\ldots+ a =0\},$$ 
so $r=q+\dots+q^{u-1}$ and $h=\deg(x)=q^{u-1}$. Then, using Theorem~\ref{teo.3.1} we have the following
\begin{proposition}
A family of Norm--Trace LRC codes has the following parameters:
$$ n=q^{2u-1}-q^{u-1}, \quad k=mr=(t+1)(q+\dots+q^{u-1})$$ and 
{\small{$$d\geq n-tq^{u-1}-(q+\dots+q^{u-1})-q^{u-1}(q^{u-1}+\dots+q-1).$$}}
\end{proposition}

For the  Norm--Trace curve $\chi$ we are able to find the $\K$-gonality of $\chi$.

\begin{lemma}\label{a2}
Let $\chi$ be a Norm--Trace curve defined over $\mathbb {F}_{q^u}$, where $u\ge 2$. We have $\gamma _{1,\mathbb {F}_{q^u}}(\chi) = q^{u-1}$.
\begin{proof}
The linear projection onto the $x$ axis has degree $q^{u-1}$ and it is defined over $\mathbb {F}_{q}$ and hence over $\mathbb {F}_{q^u}$. Thus  $\gamma_{1,\mathbb {F}_{q^u}}(\chi) \le  q^{u-1}$. Denote by $z=\gamma_{1,\mathbb {F}_{q^u}}(\chi)$ and assume that
$z\le q^{u-1}-1$. By the definition of  $\K$-gonality, there is a non-constant morphism
$w: \chi \to \mathbb {P}^1$ with $\deg (w)=z$ and defined over $\mathbb {F}_{q^u}$. Since $w(\chi(\mathbb {F}_{q^u})) \subseteq \mathbb {P}^1(\mathbb {F}_{q^u})$, we get
$\sharp (\chi(\Fqs)) \le z(q^u+1) \le (q^{u-1}-1)(q^u+1)$, that is a contradiction.
\end{proof}
\end{lemma}

\begin{remark}\label{rem.NT}
By Lemma \ref{a2}, we can apply Theorem~\ref{a1} to the Norm--Trace curve.
In fact, we can consider the gonality sequence over $\K$ of $\chi$ to get a lower bound on the second generalized Hamming weight of the two families of Norm--Trace LRC codes:
\begin{itemize}
    \item Let $t\geq 1$ and let $C$ be a $(q^{2u-1},(t+1)(q^{u-1}-1),q^{u-1}-1)$ Norm--Trace LRC code. Then we have
$$d_2(C) \ge q^{2u-1}+q^{u-1}-tq^{u-1}-(q^{u-1}-1)(1+q+\dots+q^{u-1}).$$
    \item Let $t\geq 1$ and let $C$ be a Norm--Trace LRC code with parameters $(q^{2u-1}-q^{u-1},$  $(t+1)(q+\dots+q^{u-1}),$ $q+\dots+q^{u-1})$. Then we have 
$$d_2(C) \ge q^{2u-1}-(t-1)q^{u-1}-(1+q^{u-1})(q+\dots+q^{u-1}).$$
\end{itemize}
\end{remark}

%% file: Her_Dist.tex
In this section we improve the bound on the distance of Hermitian LRC codes proposed in \cite{btv} using some properties of  \textit{Hermitian codes} which are a special case of  algebraic geometric codes.

\subsection{Hermitian codes}

Let us consider $\K=\FQ$ a finite field with $q^2$ elements. 
The \textit{Hermitian curve} $\He$ is defined over $\K$ by the affine equation
\begin{equation}\label{eqHerm}
  x^{q+1}=y^q+y\mbox{ where }x,y\in\K.
\end{equation}
This curve has genus $g=\frac{q(q-1)}{2}$ and has $q^3+1$ points of degree one, namely a pole $Q_{\infty}$ and $n=q^3$  rational affine
points, denoted by $\mathcal{P}_\He=\{ P_1,\ldots,P_n\}$  \cite{CGC-alg-art-rucsti94}.

\begin{definition}
Let $m\in\mathbb N$ such that $0\leq m\leq q^3+q^2-q-2$. Then the \textbf{Hermitian code} $C(m,q)$ is the code
$C (D, mQ_{\infty})$
where
$$D = \sum_{\alpha^{q+1}=\beta^q+\beta} P_{\alpha,\beta} $$
is the sum of all places of degree one (except $Q_{\infty}$, that is a point at infinity) of the Hermitian function field $\K(\He)$. 
\end{definition}
\noindent By Lemma 6.4.4. of \cite{CGC-cd-book-stich} we have that
$$
\Bmq=\{x^iy^j \mid qi+(q+1)j\le m,\,\,0\le i\le q^2-1,\,\,0\le j\le q-1\},
$$
forms a basis of $\mathcal{L}(mQ_{\infty})$.
For this reason, the Hermitian code $C(m,q)$ could be seen as $\mathrm{Span}_{\FQ}\langle ev_{\mathcal{P}_\He}(\Bmq)\rangle$. Moreover, the dual of $C(m,q)$ denoted by $C(m_\perp,q)=C^\perp(m,q)$ is again an Hermitian code and it is well known (Proposition 8.3.2 of \cite{CGC-cd-book-stich}) that the degree $m$ of the divisor has the following relation with respect to $m_\perp$:
\begin{equation}\label{mperp}
    m_\perp=n+2g-2-m.
\end{equation}
The Hermitian codes can be divided in four phases \cite{CGC-cd-book-AG_HB},
any of them having specific explicit formulas linking their dimension and
their distance \cite{CGC-cd-phdthesis-marcolla}. 
In particular we are interested in the first and the last phase of Hermitian codes, which are:
\begin{description}
    \item[I Phase: $0\le m_\perp\le q^2-2$]. Then we have $m_\perp=aq+b$ where $ 0\le b\le a \le q-1$ and $b\ne q-1$. In this case, the distance is 
    \begin{equation}\label{disI}
     \left\{\begin{array}{lll}
     d=a+1 & if & a>b\\
     d=a+2 & if & a=b.
     \end{array}\right. 
    \end{equation}
    \item[IV Phase: $n-1\le m_\perp \le n+2g-2$]. In this case $m_\perp=n+2g-2-aq-b$ where $a,b$ are integers such that $0\le b\le a\le q-2$ and the distance is 
     \begin{equation}\label{disII}
        d=n-aq-b.
        \end{equation}
\end{description}

\subsection{Bound on distance of Hermitian LRC codes}

Let $\K=\FQ$ be a finite field, where $q$ is a power of a prime. Let $\X=\He$ be the Hermitian curve with affine equation as in \eqref{eqHerm}.
We recall that this curve has  $q^{3}$ $\FQ$-rational affine points plus one at infinity, that we denoted by $Q_\infty$.\\
We consider two of the three constructions of Hermitian LRC codes proposed in \cite{btv} and  we improve the bound on distance of Hermitian LRC codes using properties of Hermitian codes. 
In particular, if we find an Hermitian code $C(m,q)=C_{Her}$ such that $C_{LRC}\subset C_{Her}$, then we have $d_{LRC}\geq d_{Her}$. 
\paragraph{\textbf{Projection on $\mathbf x$}} 
By Proposition 4 of \cite{btv}, we have a family of $(n,k,r)$ Hermitian LRC codes with $r=q-1$, length $n=q^3,$ dimension $k=(t-1)(q-1)$ and distance $d\geq n-tq-(q-2)(q+1)$.
Moreover, for these codes, $S=\K$, $D=tQ_\infty$ for some $1\leq t\leq q^2-1$ and the basis for the vector space $V$ is
\begin{equation}\label{base}
\mathcal B=\{x^jy^i \mid j=0,\ldots, t, \,\, i=0,\ldots,q-2\}.
\end{equation}
Using the Hermitian codes, we improve the bound on the distance for any integer $t$, such that $q^2-q+1\leq t\leq q^2-1$.\\
To find an Hermitian code $C(m,q)=C_{Her}$  such that $C_{LRC}\subset C_{Her}$, we have to compute the set $\Bmq$, that is, we have to find $m$. After that, to compute the distance of  $C(m,q)$ we use \eqref{disI} and \eqref{disII}.\\
We consider the first Hermitian phase: $0\leq m_\perp\leq q^2-2$, that is, $q^2-q+1\leq t\leq q^2-1$.\\
For this phase $m_\perp=aq+b$, where $0\leq b\leq a\leq q-1$ and the distance of the Hermitian code is either $d=  a+1$ if $a>b$ or $d=a+2$ if $a=b$. By \eqref{base}, $m$ must be equal to
$m=qt+(q+1)(q-2)$ and by \eqref{mperp} we have that $m_\perp=n+2g-2-m=q(q^2-t)$. So $b=0$ and  $a=q^2-t$ and the  distance of the Hermitian code is $d_{Her}=a+1=q^2-t+1$, since $a>b$.  This implies that
\begin{equation}\label{boundDist}
d_{LRC}\geq q^2-t+1, \mbox{ for any }  t\geq q^2-q+1.
\end{equation}
Note that  \eqref{boundDist} improves the bound on the distance proposed in Proposition 4 of \cite{btv} since
{\small{$$
q^2-t+1>q^3-tq-(q-2)(q+1) \iff t(q-1)> q(q-1)^2+1 \iff t>q^2-q.
$$}}
We just proved the following:
\begin{proposition}
Let $q^2-q+1\leq t\leq q^2-1$.
It is possible to construct a family of $(n,k,r)$ Hermitian LRC codes 
$\{C_t\}_{q^2-q+1\leq t\leq q^2-1}$ with the following parameters:
$$n=q^3,\, k=(t-1)(q-1),\, r=q-1 \mbox{ and } d\geq q^2-t+1.$$
\end{proposition}

\paragraph{\textbf{Two recovering sets}}
In \cite{btv} the authors  propose an Hermitian code with two recovering sets of size $r_1=q-1$ and $r_2=q$, denoted by  LRC(2).
They consider
$$L=Span\{x^iy^j,\,\, i=0,\ldots,q-2,\, j=0,\ldots,q-1\}$$ 
and  a linear code $C$ obtained by evaluating the functions in $L$ at the points of $B=g^{-1}(\FQ\backslash M)$, where  $g(x,y)=x$ and $M=\{a\in\Fq \mid a^{q}+ a=0 \}$. So $|B|=q^3-q$.
By Proposition 4.3 of \cite{btv}, the  LRC(2) code has length
$n=(q^2-1)q,$ dimension $k=(q-1)q$ and distance
\begin{equation}\label{dist.LRCxy}
    d\geq (q+1)(q^2-3q+3)=q^3-2q^2+3.
\end{equation}
As before, we improve the bound on the distance using Hermitian codes that contains the LRC(2) code. To do this we have to find $m_\perp$.
By $L$, we have that $m=q(q-1)+(q+1)(q-2)$ so we are in the fourth phase of Hermitian codes because $m_\perp=n+2g-2-m=q^3-q^2+q$. In this case
$d_{Her}=m_\perp-2g+2=q^3+2q+2$. Since $|B|=q^3-q$, we have that
\begin{equation}
d_{LRC}\geq d_{Her}-q = q^3+q+2.
\end{equation}
Note that this bound improves bound \eqref{dist.LRCxy}.
We just proved the following proposition:
\begin{proposition}
Let $C$ be a linear code obtained by evaluating the functions in $L$ at the points of $B$. Then $C$ has the following parameters:
$$n=(q^2-1)q,\, k=(q-1)q,\, r_1=q-1, \, r_2=q \mbox{ and } d\geq q^3+q+2.$$
\end{proposition}

%% file: ArticleArXiv.bbl
\begin{thebibliography}{10}

\bibitem{btv}
A.~Barg, I.~Tamo, and S.~Vl{\u{a}}dut.
\newblock Locally recoverable codes on algebraic curves.
\newblock {\em arXiv preprint arXiv:1501.04904}, 2015.

\bibitem{chen2007secure}
H.~Chen, R.~Cramer, S.~Goldwasser, R.~De~Haan, and V.~Vaikuntanathan.
\newblock Secure computation from random error correcting codes.
\newblock In {\em Advances in Cryptology-EUROCRYPT 2007}, pages 291--310.
  Springer, 2007.

\bibitem{7forbes2014locality}
M.~Forbes and S.~Yekhanin.
\newblock On the locality of codeword symbols in non-linear codes.
\newblock {\em Discrete Mathematics}, 324:78--84, 2014.

\bibitem{forney1994dimension}
G.~D. Forney.
\newblock Dimension/length profiles and trellis complexity of linear block
  codes.
\newblock {\em Information Theory, IEEE Transactions on}, 40(6):1741--1752,
  1994.

\bibitem{CGC-cd-art-Geil-1}
O.~Geil.
\newblock On codes from {N}orm-{T}race curves.
\newblock {\em Finite Fields Appl.}, 9:351--371, 2003.

\bibitem{geil2014relative}
O.~Geil, S.~Martin, R.~Matsumoto, D.~Ruano, and Y.~Luo.
\newblock Relative generalized {H}amming weights of one-point algebraic
  geometric codes.
\newblock {\em Information Theory, IEEE Transactions on}, 60(10):5938--5949,
  2014.

\bibitem{gopalan2011list}
P.~Gopalan, V.~Guruswami, and P.~Raghavendra.
\newblock List decoding tensor products and interleaved codes.
\newblock {\em SIAM Journal on Computing}, 40(5):1432--1462, 2011.

\bibitem{9gopalan2012locality}
P.~Gopalan, C.~Huang, H.~Simitci, and S.~Yekhanin.
\newblock On the locality of codeword symbols.
\newblock {\em Information Theory, IEEE Transactions on}, 58(11):6925--6934,
  2012.

\bibitem{guruswami2003list}
V.~Guruswami.
\newblock List decoding from erasures: {B}ounds and code constructions.
\newblock {\em Information Theory, IEEE Transactions on}, 49(11):2826--2833,
  2003.

\bibitem{heijnen1998generalized}
P.~Heijnen and R.~Pellikaan.
\newblock Generalized {H}amming weights of q-ary {R}eed-{M}uller codes.
\newblock In {\em IEEE Trans. Inform. Theory}. Citeseer, 1998.

\bibitem{CGC-cd-book-AG_HB}
T.~H{\o}holdt, J.~H. van Lint, and R.~Pellikaan.
\newblock Algebraic geometry of codes.
\newblock In V.~S. Pless and W.~Huffman, editors, {\em Handbook of coding
  theory, Vol. I, II}, pages 871--961. North-Holland, 1998.

\bibitem{homma2009second}
M.~Homma and S.~J. Kim.
\newblock The second generalized hamming weight for two-point codes on a
  {H}ermitian curve.
\newblock {\em Designs, Codes and Cryptography}, 50(1):1--40, 2009.

\bibitem{11huang2013pyramid}
C.~Huang, M.~Chen, and J.~Li.
\newblock Pyramid codes: {F}lexible schemes to trade space for access
  efficiency in reliable data storage systems.
\newblock {\em ACM Transactions on Storage (TOS)}, 9(1):3, 2013.

\bibitem{hp}
W.~C. Huffman and V.~Pless.
\newblock {\em Fundamentals of error-correcting codes}.
\newblock Cambridge university press, 2003.

\bibitem{janwa2007generalized}
H.~Janwa and A.~K. Lal.
\newblock On generalized {H}amming weights and the covering radius of linear
  codes.
\newblock In {\em Applied algebra, algebraic algorithms and error-correcting
  codes}, pages 347--356. Springer, 2007.

\bibitem{kasami1993complexity}
T.~Kasami, T.~Takata, T.~Fujiwara, and S.~Lin.
\newblock On complexity of trellis structure of linear block codes.
\newblock {\em Information Theory, IEEE Transactions on}, 39(3):1057--1064,
  1993.

\bibitem{kasami1993optimum}
T.~Kasami, T.~Takata, T.~Fujiwara, and S.~Lin.
\newblock On the optimum bit orders with respect to the state complexity of
  trellis diagrams for binary linear codes.
\newblock {\em Information Theory, IEEE Transactions on}, 39(1):242--245, 1993.

\bibitem{kurihara2011strongly}
J.~Kurihara and T.~Uyematsu.
\newblock Strongly-secure secret sharing based on linear codes can be
  characterized by generalized {H}amming weight.
\newblock In {\em Communication, Control, and Computing (Allerton), 2011 49th
  Annual Allerton Conference on}, pages 951--957. IEEE, 2011.

\bibitem{kurihara2012secret}
J.~Kurihara, T.~Uyematsu, and R.~Matsumoto.
\newblock Secret sharing schemes based on linear codes can be precisely
  characterized by the relative generalized hamming weight.
\newblock {\em IEICE Transactions on Fundamentals of Electronics, Comications
  and Computer Sciences}, 95(11):2067--2075, 2012.

\bibitem{lm}
H.~Lange and G.~Martens.
\newblock On the gonality sequence of an algebraic curve.
\newblock {\em Manuscripta mathematica}, 137(3-4):457--473, 2012.

\bibitem{luo2005some}
Y.~Luo, C.~Mitrpant, A.~J.~H. Vinck, and K.~Chen.
\newblock Some new characters on the wire-tap channel of type {II}.
\newblock {\em Information Theory, IEEE Transactions on}, 51(3):1222--1229,
  2005.

\bibitem{CGC-cd-phdthesis-marcolla}
C.~Marcolla.
\newblock {\em On structure and decoding of {H}ermitian codes}.
\newblock PhD thesis, University of Trento, 2013.

\bibitem{martin2014relative}
S.~Martin and O.~Geil.
\newblock Relative generalized {H}amming weights of q-ary {R}eed-{M}uller
  codes.
\newblock {\em arXiv preprint arXiv:1407.6185}, 2014.

\bibitem{munuera1994generalized}
C.~Munuera.
\newblock On the generalized {H}amming weights of geometric {G}oppa codes.
\newblock {\em IEEE transactions on information theory}, 40(6):2092--2099,
  1994.

\bibitem{munuera1999second}
C.~Munuera and D.~Ramirez.
\newblock The second and third generalized {H}amming weights of {H}ermitian
  codes.
\newblock {\em Information Theory, IEEE Transactions on}, 45(2):709--712, 1999.

\bibitem{ngai2011network}
C.-K. Ngai, R.~W. Yeung, and Z.~Zhang.
\newblock Network generalized hamming weight.
\newblock {\em Information Theory, IEEE Transactions on}, 57(2):1136--1143,
  2011.

\bibitem{olaya2015second}
W.~Olaya-Le{\'o}n and C.~Granados-Pinz{\'o}n.
\newblock The second generalized hamming weight of certain {C}astle codes.
\newblock {\em Designs, Codes and Cryptography}, 76(1):81--87, 2015.

\bibitem{3papailiopoulos2014locally}
D.~S. Papailiopoulos and A.~G. Dimakis.
\newblock Locally repairable codes.
\newblock {\em Information Theory, IEEE Transactions on}, 60(10):5843--5855,
  2014.

\bibitem{p}
R.~Pellikaan.
\newblock On the gonality of curves, abundant codes and decoding.
\newblock In {\em Coding theory and algebraic geometry}, pages 132--144.
  Springer, 1992.

\bibitem{p2}
B.~Poonen.
\newblock Gonality of modular curves in characteristic $ p$.
\newblock {\em Mathematical Research Letters}, 14(4):691--701, 2007.

\bibitem{CGC-alg-art-rucsti94}
H.~G. Ruck and H.~Stichtenoth.
\newblock A characterization of {H}ermitian function fields over finite fields.
\newblock {\em Journal fur die Reine und Angewandte Mathematik}, 457:185--188,
  1994.

\bibitem{7Dsilberstein2013optimal}
N.~Silberstein, A.~S. Rawat, O.~O. Koyluoglu, and S.~Vishwanath.
\newblock Optimal locally repairable codes via rank-metric codes.
\newblock In {\em Information Theory Proceedings (ISIT), 2013 IEEE
  International Symposium on}, pages 1819--1823. IEEE, 2013.

\bibitem{CGC-cd-book-stich}
H.~Stichtenoth.
\newblock {\em Algebraic function fields and codes}.
\newblock Universitext. Springer-Verlag, Berlin, 1993.

\bibitem{9Dtamo2014family}
I.~Tamo and A.~Barg.
\newblock A family of optimal locally recoverable codes.
\newblock {\em Information Theory, IEEE Transactions on}, 60(8):4661--4676,
  2014.

\bibitem{11Dtamo2013optimal}
I.~Tamo, D.~S. Papailiopoulos, and A.~G. Dimakis.
\newblock Optimal locally repairable codes and connections to matroid theory.
\newblock In {\em Information Theory Proceedings (ISIT), 2013 IEEE
  International Symposium on}, pages 1814--1818. IEEE, 2013.

\bibitem{tvn}
M.~Tsfasman, S.~Vl{\u{a}}dut, and D.~Nogin.
\newblock {\em Algebraic geometric codes: basic notions}, volume 139.
\newblock American Mathematical Soc., 1990.

\bibitem{wei1991generalized}
V.~K. Wei.
\newblock Generalized {H}amming weights for linear codes.
\newblock {\em Information Theory, IEEE Transactions on}, 37(5):1412--1418,
  1991.

\bibitem{yang1994weight}
K.~Yang, P.~V. Kumar, and H.~Stichtenoth.
\newblock On the weight hierarchy of geometric {G}oppa codes.
\newblock {\em Information Theory, IEEE Transactions on}, 40(3):913--920, 1994.

\end{thebibliography}
